\documentclass[a4paper,10pt]{extarticle}

\usepackage{amsfonts}
\usepackage{amscd,color}
\usepackage{amsmath,amsfonts,amssymb,amscd}
\usepackage{indentfirst,graphicx,epsfig}
\usepackage{graphicx}
\usepackage{tikz}
\usetikzlibrary{calc, shapes.geometric, positioning}
\usetikzlibrary{calc, decorations.pathreplacing, decorations.pathmorphing, calligraphy, shapes.geometric}
\input{epsf}
\usepackage{epstopdf}
\usepackage{caption}
\usepackage{mdframed}
\newtheorem{theorem}{Theorem}

\newtheorem{example}{Example}

\newtheorem{lemma}[theorem]{Lemma}

\newtheorem{claim}{Claim}
\newtheorem{remark}{Remark}
\newtheorem{corollary}[theorem]{Corollary}

\tikzstyle{snode}=[circle,draw=black,fill=white,thick, inner sep=0pt ,minimum size=1.2mm]
\tikzstyle{bnode}=[circle ,draw=black,fill=black,thick, inner sep=0pt ,minimum size=1.2mm]

\newenvironment {proof} {\noindent{\em Proof.}}{\hspace*{\fill}$\Box$\par\vspace{4mm}}
\newcommand{\ml}{l\kern-0.55mm\char39\kern-0.3mm}

\newenvironment{theorem-non}[1]{\trivlist \item [\hskip \labelsep {\bf #1}]\ignorespaces\it}{\endtrivlist}

\title{\textbf{Degree sum conditions for a graph to have bounded conflict-free connection number}}

\author{Dinh Hanh Dang\footnote{E-mail address: hanhdd@hau.edu.vn}\\
	Department of Mathematics\\
	Hanoi Architectural University\\
	129 TranPhu Str., Hanoi, Vietnam\\
	\medskip\\
	Trung Duy Doan\footnote{E-mail address: trungdoanduy@gmail.com}\\
	Faculty of Mathematics and Informatics\\
	Hanoi University of Science and Technology, Hanoi, Vietnam\\
	\medskip\\
	Pham Hoang Ha\footnote{E-mail address: ha.ph@hnue.edu.vn (Corresponding author).}\\
	Department of Mathematics-Informatics\\
	School of Mathematics and Computer Science\\
	Hanoi National University of Education\\
	136 XuanThuy Str., Hanoi, Vietnam\\
	\medskip\\    
	Vu Quang Minh\footnote{E-mail address: quangminhvu.nckh@gmail.com }\\
	Department of Mathematics-Informatics\\
	School of Mathematics and Computer Science\\
	Hanoi National University of Education\\
	136 XuanThuy Str., Hanoi, Vietnam}%
\date{}

\begin{document}
	\maketitle
	
	\begin{abstract}
		A path in an edge-coloured graph is called \emph{conflict-free} if a colour is exclusively applied to one of its edges. A graph $G$ is considered \emph{conflict-free connected} if every pair of vertices in $V(G)$ is connected by a conflict-free path. The minimum number of colours required to render a connected graph $G$ conflict-free connected is referred to as the \emph{conflict-free connection number}.  In this paper, we introduce several sharp conditions on the minimum degree sum of any $4$ independent vertices in $G$ to ensure that the conflict-free connection number of $G$ is bounded.
		
		
	\end{abstract}
\section{Introduction}
In this paper, we only consider finite graphs without loops or multiple edges. Let $G$ be a connected graph. We denote by $V(G)$, $E(G)$, $d_G(v)$ the vertex set, the edge set, the degree of a vertex $v$ in $G$, respectively.  For any $X\subseteq V(G)$, we
denote by $|X|$ the cardinality of $X$. Sometimes, $|G|$ is used to denote instead of $|V(G)|.$ We define $G-uv$ to be the
graph obtained from $G$ by deleting the edge $uv\in E(G)$, and $G - B$ to be the graph obtain from $G$ by deleting all edges in $B$, where $B \subseteq E(G)$. For two vertices $x$ and $y$ of $G$, the distance between $x$ and $y$ in $G$ is denoted by $d_{G}(x, y)$. For convenience, $P_G[v_1,\ldots,v_k]$ denotes the path $v_1v_2\cdots v_k$, and $C_G[v_1,\ldots,v_k]$ denotes the cycle
$v_1v_2\cdots v_kv_1$ in graph $G$. We also abbreviate the set $\lbrace 1,2,\ldots ,k\rbrace$ by $[k]$. We use \cite{West2001} for terminology and notation not defined here.

An \textit{independent set} of graph $G$ is defined as a subset $X\subseteq V(G)$ where no two vertices of $X$ are adjacent in $G$.  For an integer $m\geq 2,$ let $\alpha^{m}(G)$ denote the number defined by
\begin{center} 
	$\alpha^{m}(G)=\max\{ |S|:S\subseteq V(G),d_{G}(x,y)\geq m\, $ for all distinct vertices $ \,x,y\in S \}.$\end{center}
For an integer $p\geq 2$, we define
\begin{center}$\sigma_p^{m}(G)=\min\{\displaystyle\sum_{s\in S}d_G(s) : S\subseteq V(G), |S|=p, d_{G}(x,y)\geq m $  for all distinct vertices $ \;x,y\in S \}.$\end{center}
For convenience, we define $\sigma^{m}_{p}(G)=+\infty$ if $\alpha^{m}(G)<p$. We note that,  $\alpha^{2}(G)$ is often written as $\alpha(G)$, which is the independence number of $G,$ and $\sigma_p^{2}(G)$ is often written as $\sigma_{p}(G)$, which is the minimum degree sum of $p$ independent vertices. Moreover, we denote by $\delta(G)$ the minimum degree of $G.$

A \textit{cut-edge} of a connected graph is an edge whose deletion increases the number of components. Let $C(G)$ represent the subgraph of $G$ that is formed by the cut-edges of $G$. We say that $C(G)$ is a linear forest if each of its components is a path.  If $G$ is a connected graph of order $n$, then it is well known that $G$ contains at most $n-1$ cut-edges. The connected graphs of order $n$ with $n-1$ cut-edges must be trees. It is nontrivial to determine the maximum number of cut-edges of a connected graph. Nevertheless, many interesting results to determine the maximum number of cut-edges in connected graphs were presented. The range of the number of cut-edges in a connected graph of order $n$ and size $m$ was considered by Rao \cite{Rao1968}. Moreover, problems with additional conditions on the maximum degree and minimum degree were also considered in \cite{Rao1969}. In \cite{Achuthan2003, Suil2010}, the authors also determined the maximum number of cut-edges in a connected $d$-regular graph of order $n$.

The concept of \emph{conflict-free connection} was initially presented by Czap et al. \cite{Czap}. A path within an edge-coloured graph $G$ is labeled \emph{conflict-free} if there exists a color appearing exactly once in its edges. A graph $G$ is considered \emph{conflict-free connected} if every pair of vertices in $V(G)$ is connected by a conflict-free path. The \emph{conflict-free connection number} of $G$, denoted by $cfc(G)$, is the minimum number of colours required to make $G$ conflict-free connected. The authors in \cite{Czap} proved that if $G$ is a non-complete, $2$-connected graph, then  $cfc(G)=2$. Chang et al. \cite{Chang2021} determined the conflict-free connection number of trees. Furthermore, Huang et al. \cite{Huang2020} also demonstrated that determining the computational complexity of $cfc(G)$ for a given graph $G$ is an NP-hard problem. Many other results regarding this subject have been studied (see \cite{Chang2018, CHLMZ, Chang2019, CKP, Doan2021, Doan2024} for examples). 

One of the recent research directions on the conflict-free connection problem is determining optimal degree sum conditions such that the graph has bounded conflict-free connection number. Chang et al. \cite{Chang2018} determined graph classes with conflict-free connection number $2$ depending on the minimum degree sum of any one vertex, say $\delta(G)$ or any two non-adjacent vertices, say $\sigma_2(G)$ as follows.
\begin{theorem}[Chang et al. {\cite{Chang2018}}]
	\label{Thm_Chang_cfc_sum-degree_1_2}
	Let $G$ be a connected, non-complete graph of order $n$ and $C(G)$ is a linear forest.
	\begin{enumerate}
		\item If $n\geq 9$ and $\delta(G)\geq \max\lbrace \frac{n-4}{5}, 3\rbrace$, then $cfc(G)=2$;
		\item If $n\geq 33$ and $\sigma_2(G)\geq\frac{2n-9}{5}$, then $cfc(G)=2$.
	\end{enumerate}
\end{theorem}

Recently, Doan et al. \cite{Doan2024} improved the above results by stating the followings.
\begin{theorem}[Doan et al. {\cite{Doan2024}}]
	\label{cfc(G)=2-sigma_3(G)}
	Given an integer $n$ with $n \geq 8$, and consider a connected, non-complete graph $G$ with $n$ vertices. If $\sigma_3(G)\geq n-1$, then one of the following holds:
	\begin{enumerate}
		\item $C(G)\cong K_{1,3}$ implying $cfc(G)=3$;
		\item $C(G)$ forms a linear forest implying $cfc(G)=2$.
	\end{enumerate}
\end{theorem}
\begin{theorem}[Doan et al. {\cite{Doan2024}}]
	\label{cfc(G)=2-sigma_3(G)-natural-ext}
	Given an integer $n$ with $n\geq43$, let $G$ be a connected, non-complete graph with  $n$ vertices and $C(G)$ forms a linear forest. If $\delta(G)\geq3$ and $\sigma_3(G)\geq\frac{3n-14}{5}$, then $cfc(G)=2$.
\end{theorem}

To establish the aforementioned results, the authors often limit the number of cut-edges in a graph and subsequently reduce the problem to determining its conflict-free connection number. However, in the proofs, they all divided the proof steps into many subcases and their arguments depend heavily on whether the number of independent vertices is small or large. As a result, it seems to be very difficult to extend these methods from the case of $\sigma_3(G)$ to $\sigma_4(G)$. In this paper, we introduce a new approach to resolve that problem. We believe that it provides a useful framework to consider the general case concerning the degree sum of $p$ independent vertices where $p\geq 5.$ In particular, we prove the following theorems for the main results of this paper.

First, we state the condition for a connected graph to have bounded conflict-free connection number.
\begin{theorem}{\label{main0}} 
	Let $G$ be a connected graph of order $n$ and let $k \ge 3$ be an integer. If
	\[
	\sigma_{4}(G) \ge
	\begin{cases} 
		2n-7 &\text{for } k=3 \text{ and } n\ge 8 \\ 
		n &\text{for } k =4 \text{ and } n\ge9\\
		n-1 &\text{for } k=5 \text{ and } n \ge 11 \\
		\displaystyle\frac{n+5}{2} &\text{for } k=6 \text{ and } n \ge 17 \\
		\displaystyle\frac{4n-4k-7}{k+2} & \text{for } k \ge 7 \text{ and }\begin{cases} 
			n \ge 286 &\text{for } k=7 \\
			n \ge 207 &\text{for } k=8 \\
			n \ge 190 &\text{for } k=9 \\
			\displaystyle n > k^2 + 7k + 18 + \frac{33}{k-2}  &\text{for } k\ge 10, \\
		\end{cases}
	\end{cases}
	\]
	then $cfc(G)\leq k.$
\end{theorem}

Finally, we consider the conditions for a graph to have small conflict-free connection number.
\begin{theorem}\label{main1}
	Let $n\geq 8$ be an integer and $G$ be a connected, non-complete graph of order $n$. If $\sigma_4(G)\geq 2n-7$, then one of the following holds:
	\begin{enumerate}
		\item $C(G)\cong K_{1,3}$ implying $cfc(G)=3$.
		\item $C(G)$ is a linear forest implying $cfc(G)=2$.
	\end{enumerate}
\end{theorem}
\begin{theorem}\label{main2}
	Let $n\geq9$ be an integer and $G$ be a connected, non-complete graph of order $n$ and $\delta(G)\geq2$. If $\sigma_4(G)\geq n$, then $cfc(G)=2$.
\end{theorem}
\begin{theorem}\label{main3}
	Let $n\geq 21$ be an integer and $G$ be a connected, non-complete graph of order $n$ and $\delta(G)\geq 3$. If $C(G)$ is a linear forest and $\sigma_4(G)\geq\frac{4n-19}{5}$, then $cfc(G)=2$.
\end{theorem}

This paper is organized as follows: In Section 2, we prove a result on the degree sum conditions and the number of cut-edges in a connected graph. In Section 3, we focus on studying the degree sum conditions for a connected graph to have bounded conflict-free connection number. In particular, we prove Theorems \ref{main1}-\ref{main3} and show their sharpness. In the last section, we introduce some discussions regarding our results and related topic.
\section{The minimum degree sum and the number of cut-edges}
In this section, we first recall the condition on the minimum degree sum of any $3$ independent vertices for a graph to have few cut-edges.
\begin{theorem}[Doan et al.{\cite{Doan2024}}]
	\label{thm_l=3}
	Consider two integers $n$, $k$, where $k\geq3$. Let $G$ be a connected graph of order $n$. If 
	$$
	\sigma_3(G)\geq
	\begin{cases}
		n-1 &\text{for } k = 3 \text{ and } n \geq 8 \\
		\displaystyle\frac{n+3}{2} &\text{for } k = 4 \text{ and } n \geq 13 \\
		\displaystyle\frac{3n-3k-5}{k+2} &\text{for } k\geq 5 \text{ and }
		\begin{cases}
			n \geq 125 \text{ for } k=5\\
			 n \geq 102 \text{ for } k=6\\
			n\geq k^2+6k+13 \text { for } k\geq 7
		\end{cases}
	\end{cases}
	$$
	then $G$ has at most $k$ cut-edges.
\end{theorem}

The main goal of this section is to prove the following.
\begin{theorem} \label{thm0.1} 
Consider two integers $n, k$, where $k\ge3$. Let $G$ be a connected graph of order $n$. If
\[
\sigma_{4}(G) \ge
\begin{cases} 
2n-7 &\text{for } k=3 \text{ and } n\ge 8 \\ 
n &\text{for } k =4 \text{ and } n\ge9\\
n-1 &\text{for } k=5 \text{ and } n \ge 11 \\
\displaystyle\frac{n+5}{2} &\text{for } k=6 \text{ and } n \ge 17 \\
\displaystyle\frac{4n-4k-7}{k+2} & \text{for } k \ge 7 \text{ and }\begin{cases} 
n \ge 286 &\text{for } k=7 \\
n \ge 207 &\text{for } k=8 \\
n \ge 190 &\text{for } k=9 \\
\displaystyle n > k^2 + 7k + 18 + \frac{33}{k-2}  &\text{for } k\ge 10 \\
\end{cases}
\end{cases}
\]
then $G$ has at most $k$ cut-edges.
\end{theorem}

Before proceeding to the proof, we would like to state a relationship between the results of Theorem \ref{thm0.1} and Theorem \ref{thm_l=3} by the following remark. 
\begin{remark}{}{}  
\upshape 
    Note that:
    \begin{itemize}
        \item For $k=3$, we have $\displaystyle 2n-7 > \frac{4n-4k-7}{k+2}$;
        \item For $k=4$, we have $\displaystyle n > \frac{4n-4k-7}{k+2}$;
        \item For $k=5$, we have $\displaystyle n-1 > \frac{4n-4k-7}{k+2}$;
        \item For $k=6$, we have $\displaystyle \frac{n+5}{2} > \frac{4n-4k-7}{k+2}$.
    \end{itemize}

    Moreover, if the lower bound of the minimum degree sum of any three independent vertices is $\sigma_{3}(G)  \ge\frac{3n-3k-5}{k+2}$, then the lower bound of the minimum degree sum of any four arbitrary independent vertices will be:
    \[
    \sigma_{4}(G) \ge \frac{4}{3}\sigma_{3}(G) \ge \frac{4}{3}\left(\frac{3n-3k-5}{k+2}\right) > \frac{4n-4k-7}{k+2}.
    \]
\end{remark}

Now, we prove Theorem \ref{thm0.1}.

\begin{proof}
Suppose for the sake of contradiction that graph $G$ has at least $k+1$ cut-edges. Let $B$ be the set of $k+1$ cut-edges of $G$. Then $G - B$ has exactly $k+2$ components, say $G_{1},G_2,\ldots,G_{k+2}.$ 

Let $S = \{G_\alpha \in G - B \mid N(v_\alpha) \not \subseteq V(G_\alpha) \text{ for all } v_\alpha \in V(G_\alpha)\}$. We consider the following cases based on the cardinality of the set $S$.

\noindent\textbf{Case 1.} $\lvert S \rvert=0$. 

In this case, for every index $j$ in $[k+2]$, there exists a vertex $v_{j}\in V(G_{j})$ such that $N(v_{j})\subseteq V(G_{j})$. Clearly, $\lvert V(G_{j}) \rvert \ge d_{G}(v_{j})+1$ for all $j \in [k+2]$. If $G_{x}$, $G_{y}$, $G_{z}$, $G_t$ are four distinct components of $G - B$, then for $v_{i}\in V(G_{i}), i\in\{x,y,z,t\}$ such that $N(v_i) \subseteq V(G_i)$, we get that $\{v_{x},v_{y},v_{z}, v_t\}$ is an independent set. Therefore
\[
\sum_{i\in\{x,y,z,t\}}\lvert V(G_{i})\rvert\ge\sum_{i\in\{x,y,z,t\}}d_{G}(v_{i})+4 \ge \sigma_{4}(G)+4.
\]
From here, we have an estimation:
$$4n=4\sum_{j=1}^{k+2}\lvert V(G_{j}) \rvert \ge (k+2)\left(\frac{4n-4k-7}{k+2}+4\right)=4n+1.$$
Obviously, this is a contradiction for all $n$.

For $\lvert S \rvert = m\geq 1,$ without loss of generality, we may assume that $S=\{G_1,G_2\ldots, G_m\}.$

\noindent\textbf{Case 2.} $m=1$. 

Then we obtain that $N(v)\not\subseteq V(G_{1})$ for all vertices $v\in V(G_{1})$. Let $a=\lvert V(G_{1}) \rvert\ge1$. Since each vertex in $V(G_1)$ is adjacent to at least one cut-edge from the set $B$, it follows that $a\le \lvert B \rvert = k+1$. There always exists a vertex $v_{1}\in V(G_{1})$ such that 
\[
d_{G}(v_{1})\le \frac{\lvert B \rvert}{a}+(a-1)=\frac{k+1}{a}+a-1.
\]
Since $1\le a\le k+1$, we deduce $d_{G}(v_{1})\le k+1$. For any three distinct components $G_{x}$, $G_{y}$, $G_z$ ($x,y,z \in [k+2]\setminus [1]$), there exist three vertices $v_{x}\in V(G_{x})$, $v_{y}\in V(G_{y})$, $v_{z}\in V(G_{z})$ such that $N(v_{x})\subseteq V(G_{x})$, $N(v_{y})\subseteq V(G_{y})$ and $N(v_{z})\subseteq V(G_{z})$. Moreover, it can be witnessed that $\{v_{1},v_{x},v_{y}, v_z\}$ forms an independent set, so $\displaystyle \sum_{i \in \{x,y,z\}} d_G(v_i) \ge\sigma_{4}(G)-(k+1)$. This implies 
\begin{align*}
\sum_{i \in \{x,y,z\}} \lvert V(G_{i}) \rvert \ge \sum_{i \in \{x,y,z\}} d_G(v_i) + 3 \ge \sigma_{4}(G)-(k+1)+3 = \sigma_4(G) - k+2.
\end{align*}
Now, we have
\begin{align*}
3n &= 3\sum_{j=1}^{k+2}\lvert V(G_{j})\rvert \ge 3\left(\sum_{j=2}^{k+2}\lvert V(G_{j})\rvert\right)+3\ge(k+1)(\sigma_{4}(G)-k+2)+3.
\end{align*}
For $k=3,k=4,k=5,k=6, k=7,k=8,k=9$, we obtain 
\[
n \le \frac{29}{5}, n\le \frac{7}{2}, n \le 7, n\le 15, n \le \frac{613}{5}, n\le \frac{287}{2}, n \le \frac{1167}{7},
\]
respectively, a contradiction. 
For $k \ge 10$, 
\[
n \le k^2+7k+18+\frac{33}{k-2}.
\]
This is also a contradiction. 

\noindent\textbf{Case 3. }$m=2$. 

Now, we consider two subcases as follows:

\textbf{Case 3.1. }If there exist two vertices $u \in V(G_1)$ and $v \in V(G_2)$ such that $uv$ is not an edge of the graph $G$, then $d_G(u) + d_G(v) \le \lvert B \rvert + 1= k+2$ since every vertex in $V(G_1),V(G_2)$ is incident to at least one cut-edge in $B$. For any components $G_t,G_s$ ($t,s \in [k+2]\setminus [2]$), there exist two vertices $v_t \in V(G_t), v_s \in V(G_s)$ such that $N(v_t) \subseteq V(G_t), N(v_s) \subseteq V(G_s)$. Obviously, $\{u,v,v_t,v_s\}$ is an independent set. Therefore $d_G(u) + d_G(v) + d_G(v_t) + d_G(v_s) \ge \sigma_4(G)$. Hence, 
\begin{align*}
    \lvert V(G_t) \rvert + \lvert V(G_s) \rvert &\ge (d_G(v_t) + 1) + (d_G(v_s) + 1) \\ 
    &\ge \sigma_4(G) - d_G(u) -d_G(v) +2 \ge \sigma_4(G) - (k+2) +2 =\sigma_4(G)-k.
\end{align*}
From here, we have an estimation
\begin{align*}
2n &= 2\sum_{j=1}^{k+2}\lvert V(G_j)\rvert = 2(\lvert V(G_1)\rvert + \lvert V(G_2)\rvert) + 2\sum_{j=3}^{k+2} \lvert V(G_j)\rvert \ge 4 + k(\sigma_4(G)-k). 
\end{align*}
For $k=3,k=4,k=5,k=6, k=7,k=8,k=9$, we have 
\[
n \le \frac{13}{2}, n\le 6, n \le \frac{26}{3},n\le 17, n\le 65, n \le 76, n\le \frac{617}{7},
\]
respectively, a contradiction. 
For $k \ge 10$, 
\[
n \le \frac{1}{2}k^2+4k+\frac{19}{2}+\frac{15}{k-2} < k^2+7k+18+\frac{33}{k-2}.
\]
This is also a contradiction. 

\textbf{Case 3.2. }If $uv$ is an edge of the graph $G$ for all $u \in V(G_1)$ and $v \in V(G_2)$. Then, since there is at most one cut-edge connecting $G_1$ with $G_2$, we must have $\lvert V(G_1) \rvert = \lvert V(G_2) \rvert = 1$. Assume $V(G_1) = \{u\}, V(G_2) = \{v\}$. Thus, we obtain $d_G(u) + d_G(v) \le k+2$. Without loss of generality, suppose that $d_G(u) \le \frac{k+2}{2}$. For each triple of components in $G - B$, say $G_x, G_y, G_z$ with $x,y,z \in [k+2]\setminus[2]$, there exist three vertices $v_x \in V(G_x), v_y \in V(G_y), v_z \in V(G_z)$ such that $N(v_x) \subseteq V(G_x), N(v_y) \subseteq V(G_y), \text{ and } N(v_z) \subseteq V(G_z)$. It is easy to see that $\{u, v_x, v_y, v_z\}$ is an independent set, so $d_G(u) + d_G(v_x) + d_G(v_y) + d_G(v_z) \ge \sigma_4(G)$. Therefore
\[
 \sum_{i \in \{x,y,z\}} \lvert V(G_i) \rvert \ge \sum_{i \in \{x,y,z\}} d_G(v_i) + 3 \ge \sigma_4(G) - d_G(u) + 3 \ge \sigma_4(G) - \frac{k+2}{2} + 3.
\]
Thus, we have 
\begin{align*}
3n =& 3\sum_{j=1}^{k+2} \lvert V(G_j)\rvert = 3\sum_{j=3}^{k+2} \lvert V(G_j)\rvert + 3(\lvert V(G_1)\rvert+\lvert V(G_2)\rvert) \ge k\left(\sigma_4(G) - \frac{k+2}{2} + 3\right)+6.
\end{align*}
For $k=3,k=4,k=5,k=6, k=7,k=8,k=9$, we obtain 
\[
n \le \frac{9}{2}, n\le -6, n \le \frac{3}{4}, 3n+15\le 3n, n \le \frac{571}{2}, n\le 206, n \le \frac{379}{2},
\]
respectively, a contradiction. For $k \ge 10$, 
\[
n \le \frac{1}{2}k^2+6k+33+\frac{186}{k-6} < k^2+7k+18+\frac{33}{k-2}.
\]
This is also a contradiction. 

\noindent\textbf{Case 4. }$m=3$. 

\textbf{Case 4.1. }There exist three vertices $v_1 \in V(G_1), v_2 \in V(G_2), v_3 \in V(G_3)$ such that $\{v_1, v_2, v_3\}$ forms an independent set. Note that, there are no cut-edges joining all three pairs of these vertices since otherwise a cycle would be formed, so there are at most two cut-edges joining these three vertices. Therefore we have $\displaystyle\sum_{i \in [3]} d_G(v_i) \le \lvert B \rvert + 2 = k+3$. For any component $G_t$ ($t \ge 4$) from $k-1$ remaining components, we can always choose a vertex $v_t \in V(G_t)$ such that $N(v_t) \subseteq V(G_t)$. Clearly, $\{v_1, v_2, v_3, v_t\}$ is an independent set, so $\displaystyle \sum_{i \in [3]} d_G(v_i) + d_G(v_t) \ge \sigma_4(G)$. Therefore, 
\[
    \lvert V(G_t) \rvert \ge d_G(v_t) + 1 \ge \sigma_4(G) - \sum_{i=1}^3 d_G(v_i) + 1 \ge \sigma_4(G) - (k+3)+1 = \sigma_4(G) - k - 2.
\]
Thus, the number of vertices in graph $G$ is
\begin{align*}
n &= \sum_{j=1}^{k+2} \lvert V(G_j)\rvert = \sum_{j=1}^3 \lvert V(G_j)\rvert + \sum_{j=4}^{k+2} \lvert V(G_j)\rvert\ge 3 + (k-1)(\sigma_4(G) - k - 2).
\end{align*}
For $k=3, k=4, k=5, k=6, k=7,k=8, k=9$, we obtain 
\[
n \le 7, n\le \frac{15}{2}, n\le \frac{29}{3}, n \le \frac{49}{3}, n \le \frac{223}{5}, n \le \frac{943}{18}, n \le \frac{1279}{21},
\]
respectively, a contradiction. For $k \ge 10$, we have 
\[
n \le \frac{1}{3}k^2+3k+6+\frac{19}{3(k-2)} < k^2+7k+18+\frac{33}{k-2}.
\]
This is also a contradiction. 

\textbf{Case 4.2. }For any choice of three vertices $v_1 \in V(G_1), v_2 \in V(G_2), v_3 \in V(G_3)$, $\{v_1, v_2, v_3\}$ is not an independent set. Nevertheless, we can always choose two non-adjacent vertices, say $v_1, v_3$, in $S$ such that $d_G(v_1)+d_G(v_3) \le \lvert B \rvert=k+1$. For two vertices $v_t \in V(G_t), v_s \in V(G_s)$ such that $N(v_t) \subseteq V(G_t), N(v_s) \subseteq V(G_s)$ with $t,s \in [k+2]\setminus[3]$, we have $\{v_1, v_3, v_t, v_s\}$ forms an independent set. Therefore $d_G(v_1) + d_G(v_3) + d_G(v_t) + d_G(v_s) \ge \sigma_4(G)$, which leads to 
\begin{align*}
    \lvert V(G_t) \rvert+\lvert V(G_s) \rvert &\ge (d_G(v_t)+1)+(d_G(v_s)+1) \ge \sigma_4(G) - d_G(v_1) - d_G(v_3) + 2 \\ 
    &\ge \sigma_4(G)-(k+1)+2=\sigma_4(G)-k+1.
\end{align*}
Thus, we have 
\begin{align*}
2n = &2\sum_{j=1}^{k+2} \lvert V(G_j)\rvert =2\left(\sum_{j=1}^3 \lvert V(G_j)\rvert + \sum_{j=4}^{k+2} \lvert V(G_j)\rvert\right) \ge 6 + (k-1)(\sigma_4(G)-k+1).
\end{align*}
For $k=3, k=4, k=5, k=6, k=7, k=8, k=9$, we obtain
\[
n \le 6, n \le 3, n\le 7, n\le 13, n \le 80, n\le \frac{703}{8}, n\le \frac{491}{5},
\]
respectively, a contradiction. For $k \ge 10$, we have 
\[
n \le \frac{1}{2}k^2+4k+13+\frac{87}{2(k-4)} < k^2+7k+18+\frac{33}{k-2}.
\]
This is also a contradiction.

\noindent\textbf{Case 5. }$m \ge 4$. 

To prove this case, we first state the following claims.

\begin{claim}\label{claim1}
    If $G_x, G_y, G_z, G_t$ are four distinct components in $S$, then $\{v_x,v_y, v_z, v_t\}$ is not an independent set for all $v_x \in G_x,v_y \in G_y, v_z \in G_z, v_t \in G_t$, and consequently there are at most three edges among these four vertices. 
\end{claim}
\begin{proof}
    Suppose to the contrary that $\{v_x,v_y, v_z, v_t\}$ is an independent set. Then 
    \[
    \sum_{i \in \{x,y,z,t\}} d_G(v_i) \ge \sigma_4(G).
    \]
    On the other hand, using a similar argument as in Case 4.1, we obtain that
    \[
    \sum_{i \in \{x,y,z,t\}} d_G(v_i) \le \lvert B \rvert+3 = k+4.
    \]
    Thus, $k+4 \ge \sigma_4(G)\ge \frac{4n-4k-7}{k+2}$. For $k=3,k=4,k=5,k=6,k=7,k=8,k=9$, we obtain
    \[
    n \le 7, n \le 8, n\le 10, n\le 15, n\le \frac{67}{2}, n \le \frac{159}{4}, n\le \frac{93}{2},
    \]
	a contradiction. For $k \ge 10$, we have 
    \[
    n \le \frac{1}{4}k^2+\frac{5}{2}k+\frac{15}{4} < k^2+7k+18+\frac{33}{k-2}.
    \]
    This is a contradiction. Moreover, since $v_x, v_y, v_z, v_t$ cannot form a cycle, there are at most three edges among them. This completes the proof of Claim \ref{claim1}.
\end{proof}
\begin{claim}\label{claim2}
    For every component $G_\alpha \in S$, there are at most three cut-edges connecting $G_\alpha$ with other components of $S$.
\end{claim}

\begin{proof}
    Suppose to the contrary that there are more than three cut-edges connecting $G_\alpha$ with other components of $S$. Then there exist distinct components $G_x, G_y, G_z, G_t$ in $S$ and cut-edges $e_x, e_y, e_z, e_t$ connecting $G_\alpha$ to $G_x, G_y, G_z, G_t$, respectively.
    Assume $e_x = u_{\alpha x}v_x$, $e_y = u_{\alpha y}v_y$, $e_z = u_{\alpha z}v_z$, $e_t = u_{\alpha t}v_t$ with $u_{\alpha x}, u_{\alpha y}, u_{\alpha z}, u_{\alpha t} \in V(G_\alpha)$ and $v_x \in V(G_x), v_y \in V(G_y), v_z \in V(G_z), v_t \in V(G_t)$. Note that the vertices $u_{\alpha x}, u_{\alpha y}, u_{\alpha z}, u_{\alpha t}$ are not necessarily distinct. It can be witnessed that the set $\{v_x, v_y, v_z, v_t\}$ is independent. Indeed, otherwise, we can assume $v_x v_y \in E(G)$. Then this forms a cycle in $G$: $C_G[u_{\alpha x}, v_x, v_y, u_{\alpha y}]$ where $u_{\alpha x}$ is connected to $u_{\alpha y}$ via $P_{G_\alpha}[u_{\alpha x}, u_{\alpha y}]$. This implies that $e_x$ is not a cut-edge of $G$, which is a contradiction. Thus, $\{v_x, v_y, v_z, v_t\}$ is an independent set. By Claim 1, this leads to a contradiction.
    
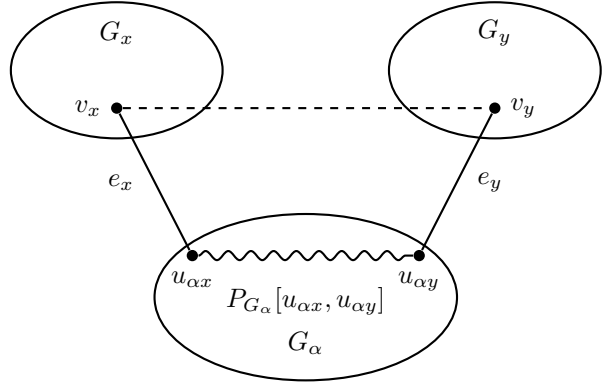
\begin{center}
        \begin{tikzpicture}[
    vertex/.style={circle, fill=black, inner sep=1.5pt},
    component/.style={draw, thick, ellipse, minimum width=2.8cm, minimum height=1.8cm},
    main_component/.style={draw, thick, ellipse, minimum width=4cm, minimum height=2.2cm}
]
    \node[main_component] (G_alpha) at (0, 0) {};
    \node[font=\bfseries] at (0, -0.6) {$G_\alpha$};
    \node[component] (G_x) at (-2.5, 3) {};
    \node[font=\bfseries] at (-2.5, 3.5) {$G_x$}; 
    \node[component] (G_y) at (2.5, 3) {};
    \node[font=\bfseries] at (2.5, 3.5) {$G_y$}; 
    \node[vertex, label=below:{$u_{\alpha x}$}] (u_x) at (-1.5, 0.55) {};
    \node[vertex, label=below:{$u_{\alpha y}$}] (u_y) at (1.5, 0.55) {};
    \node[vertex, label=left:{$v_x$}] (v_x) at (-2.5, 2.5) {};
    \node[vertex, label=right:{$v_y$}] (v_y) at (2.5, 2.5) {};
    \draw[thick, decorate, decoration={snake, segment length=3mm, amplitude=0.6mm}] 
        (u_x) to node[below=8pt] {$P_{G_\alpha}[u_{\alpha x}, u_{\alpha y}]$} (u_y);
    \draw[thick] (u_x) -- (v_x) node[midway, left=4pt, font=\bfseries] {$e_x$};
    \draw[thick] (u_y) -- (v_y) node[midway, right=4pt, font=\bfseries] {$e_y$};
    \draw[thick,dashed] (v_x) -- (v_y) node[midway, above=2pt, font=\bfseries] {};
\end{tikzpicture}
    \end{center}
\end{proof}
\begin{claim}\label{claim3}
The cardinality of the set $S$ is at most $k+1$.
\end{claim}
\begin{proof}
    Suppose for the sake of contradiction that $m = k+2$. We contract each component $G_\alpha \in S$ into a single vertex $u_\alpha$, while keeping the set of cut-edges $B$ to connect these vertices. Since $B$ is a set of cut-edges, the resulting graph is a tree $T$ of order $k+2$. We consider the following cases based on the number of leaves, denoted by $|L(T)|$ of the tree $T$:

    \noindent\textbf{Case C1. $|L(T)| \ge 4$.}
    In this case, the leaves of $T$ form an independent set of size at least 4. This directly contradicts Claim 1.

    \noindent\textbf{Case C2. $|L(T)| = 2$.}
    Then the tree $T$ has no branch vertices and hence is a path.
    
    \textbf{Subcase C2.1. $k \ge 5$.}
        We have $m = k+2 \ge 7$. Now we can always choose 4 pairwise non-adjacent vertices in $T$. By Claim 1, this leads to a contradiction.
        
    \textbf{Subcase C2.2. $k = 4$.}
        We have $m = k+2 = 6$. Suppose $T$ is the path $P_T[u_1, u_2, u_3, u_4, u_5, u_6]$. 
        If $|G_1| + |G_6| \ge 3$, assume without loss of generality that $|G_6| \ge 2$. Then there exists a vertex $v_6 \in V(G_6)$ such that $N_G(v_6) \subseteq V(G_6)$, which contradicts the property of the set $S$. 
        Thus, $|G_1| = |G_6| = 1$. Then we have $(|G_2| + |G_4|) + (|G_3| + |G_5|) = n - 2 \ge 7$. Assume $|G_3| + |G_5| \ge 4$. Since $G_3$ and $G_5$ are incident to exactly two cut-edges, every vertex of $G_3$ and $G_5$ must be incident to a cut-edge, we deduce that $|G_3| = |G_5| = 2$. Hence, we obtain an independent set consisting of 4 vertices belonging to distinct components of $S$, contradicting Claim 1.

    \textbf{Subcase C2.3. $k = 3$.}
        We have $m = 5$ and $T$ is the path $P_T[u_1, u_2, u_3, u_4, u_5]$. By similar arguments, we obtain $|G_1| = |G_5| = 1$ and $|G_2| = |G_3| = |G_4| = 2$, which again contradicts Claim 1.

    \noindent\textbf{Case C3. $|L(T)| = 3$.}
    The tree $T$ has exactly one branch vertex $u$, corresponding to $G_u \in S$, with $d_T(u) = 3$. Let $u_1, u_2, u_3$ be the 3 leaves of $T$, where $|G_{u_i}| = 1$ for each $i \in [3]$. 
    If the distance from $u$ to all leaves is at least 2, then $\{u, u_1, u_2, u_3\}$ is an independent set of 4 vertices, a contradiction. Therefore, without loss of generality, assume that $d_T(u, u_1) = 1$ and 
    \begin{equation}
        d_T(u, u_2) \le d_T(u, u_3).
    \end{equation}
    We have
    \begin{equation}
        d_T(u, u_2) + d_T(u, u_3) = k+1 - d_T(u, u_1) = k.
    \end{equation}
    From (1) and (2), we obtain $d_T(u, u_3) \ge 2$. Hence,  $P_T (u,u_3) = P_T [u,u_3] \setminus \{u , u_3\} \ne \emptyset$. Let $v = N_T(u) \cap P_T [u,u_3]$.

    \noindent\textbf{Subcase C3.1.} $|G_u| = 1.$ Then $ V(G_u) = \{u\}$.
\begin{itemize}
    \item If $k = 3$, then from (2), we have $d_T(u, u_2) = 1$ and $d_T(u, u_3) = 2$. Therefore $|G_v| = n - |\{u_1, u_2, u_3, u\}| = n - 4 \ge 4$, a contradiction since $G_v$ is incident to exactly 2 cut-edges.
    
    \item If $k = 4$, then we have $d_T(u, u_2) = 1$ and $d_T(u, u_3) = 3$. It follows that $\{u_1, u_2, u_3, v\}$ is an independent set, contradicting Claim 1. Thus, we consider $d_T(u, u_2) = d_T(u, u_3) = 2$. Let $w \in V(P_T(u, u_2))$. We have $|G_v| + |G_w| = n - 4 \ge 5$, implying $|G_v| \ge 3$ or $|G_w| \ge 3$, a contradiction.
    
    \item If $k \ge 5$, then combining with (2), we obtain $d_T(u, u_3) \ge 3$. This implies $\{u_1, u_2, u_3, v\}$ is independent, contradicting Claim 1.
\end{itemize}

\noindent\textbf{Subcase C3.2.} $|G_u| = 2$. Let $V(G_u) = \{u_{11}, u_{12}\}$ with $u_{1}u_{11} \in E(G)$.

\textbf{Subcase C3.2.1.} $d_T(u, u_2) = 1$. If $k \ge 4$, then $d_T(u, u_3) = k - d_T(u, u_2) \ge 3$. Therefore the set $\{u_{1}, u_2, u_3, v\}$ is independent, contradicting Claim 1. If $k = 3$, then $d_T(u, u_3) = 2$. This implies that $|V(G_v)| = n - |\{u_1, u_2, u_3, u_{11}, u_{12}\}| = n - 5 \ge 3$, a contradiction.

\textbf{Subcase C3.2.2.} $d_T(u, u_2) \ge 2$. This implies $d_T(u, u_3) \ge d_T(u, u_2) \ge 2$. Thus $\{u_1, u_2, u_3, u_{12}\}$ is an independent set, contradicting Claim 1.

\noindent\textbf{Subcase C3.3.} $|G_u| = 3$. Assume $V(G_u) = \{u_{11}, u_{12}, u_{13}\}$, where $u_{11}u_{1} \in E(G)$, \(u_{12}\) is the vertex of \(G_u\) closest to \(u_2\), and
\(u_{13}\) is the vertex of \(G_u\) closest to \(u_3\). By (2), we deduce $2 d_T(u, u_3) \ge d_T(u, u_2) + d_T(u, u_3) = k \ge 3$. Hence $d_T(u, u_3) \ge 2$. Therefore $\{u_1, u_2, u_3, u_{13}\}$ is an independent set, contradicting Claim 1.

This completes the proof of Claim 3.
\end{proof}

Return to Case 5. According to the above claims, we have that each vertex in $V(G_\alpha)$ with $G_\alpha \in S$ is adjacent to at most three cut-edges of $G$, so if we contract each component $G_\alpha$ of $S$ into a single vertex and keep all the cut-edges connecting components of $S$, the resulting graph $N$ is an induced subgraph of $G$ on the set $S$. Then $N$ will fall into one of the following two possibilities: 

\noindent\textbf{Case 5.1. }$\Delta(N) =2$. 

In this case, $N$ is a linear forest of order $m$. If $m \ge 7$, we can easily choose four vertices in $N$ to form an independent set, which is a contradiction to Claim 1. Therefore we must have $m \le 6$, combining with the hypothesis $m \ge 4$, we get $4 \le m \le 6$. 

\textbf{Case 5.1.1. }$\alpha(N) = 3$. 

We can categorize the configurations of $N$ based on the value of $m$ as follows:

If $m=4$, then $N \in \{P_3 \cup K_1\}$.

If $m=5$, then $N \in \{P_5, P_3 \cup P_2, P_4 \cup P_1\}$.

If $m=6$, then $N \in \{P_6, 3P_2, P_4 \cup P_2\}$.

It can be seen that we can always choose three non-adjacent vertices $v_i \in V(G_i)$ with $i \in \{x,y,z\} \subset [m]$ such that $\displaystyle\sum_{i \in \{x,y,z\}} d_G(v_i) \le \lvert B \rvert = k+1$ and $\{v_x, v_y, v_z, v_t\}$ forms an independent set with $v_t \in V(G_t), t \in [k+2]\setminus[m]$. Then, $\displaystyle\sum_{i \in \{x,y,z\}}d_G(v_i)+d_G(v_t) \ge \sigma_4(G)$. It follows that 
\[
\lvert V(G_t) \rvert \ge d_G(v_t) + 1 \ge \sigma_4(G)  - \displaystyle\sum_{i \in \{x,y,z\}} d_G(v_i) + 1 \ge \sigma_4(G) - (k+1) +1 =\sigma_4(G)-k
\]
for all $t \in [k+2]\setminus[m]$. 

$\bullet$ For $k=3$, we have $\lvert V(G_t) \rvert \ge \sigma_4(G)-3 \ge (2n-7)-3 = 2n-10$ for all $t \in [5]\setminus [m]$, since $\sigma_4(G) \ge 2n-7$. By Claim \ref{claim3}, we obtain $m=4$. Then, $n \ge 4.1 + (2n-10)=2n-6>n$ for all $n \ge 8$, a contradiction. 

$\bullet$ For $k=4$, we have $\lvert V(G_t) \rvert \ge \sigma_4(G)-4 \ge n-4$ for all $t \in [6]\setminus [m]$, since $\sigma_4(G) \ge n$. By Claim \ref{claim3}, we get $m\in \{4;5\}.$

If $m=4$, then $n  \ge 4.1 + 2(n-4) = 2n-4 >n$, a contradiction. 

If $m=5$, then $n \ge 5.1 + (n-4) = n+1 >n$, a contradiction. 

$\bullet$ For $k=5$, we have $\lvert V(G_t) \rvert \ge \sigma_4(G)-5 \ge (n-1)-5 = n-6$ for all $t \in [7]\setminus [m]$, since $\sigma_4(G) \ge n-1$. 

If $m=4$, then $n \ge 4.1 + 3(n-6) = 3n-14 > n$, a contradiction. 

If $m=5$, then $n \ge 5.1 + 2(n-6)=2n-7 >n$, a contradiction. 

If $m=6$, then $\lvert V(G_i) \rvert=1$ for all $i \in [6]$. For $i\in \{1,2\}$ and $v_7 \in V(G_7)$ such that $N(v_7) \subseteq V(G_7)$, $\{v_i, v_{i+2}. v_{i+4}, v_7\}$ is an independent set and hence
\[
d_G(v_i) + d_G(v_{i+2}) + d_G(v_{i+4} )+d_G(v_7) \le 1+2+2+(n-6-1) = n-2 <n-1,
\]
a contradiction (see graph $H_3$ in Example 3).

$\bullet$ For $k=6$, we have $\lvert V(G_t) \rvert \ge \sigma_4(G) - 6 \ge \frac{n+5}{2} - 6$ for all $t \in [8] \setminus [m]$ since $\sigma_4(G) \ge \frac{n+5}{2}$.

If $m=4$, then $n \ge 4.1 + 4\left(\frac{n+5}{2}-6\right) =2n-10 > n$, a contradiction. 

If $m=5$, then $n \ge 5.1 + 3\left(\frac{n+5}{2}-6\right) =\frac{3n-11}{2}> n$, a contradiction. 

If $m=6$, then $\lvert V(G_7) \rvert+\lvert V(G_8) \rvert = n-6$. Without loss of generality, we assume $\lvert V(G_7) \rvert \le \frac{n-6}{2}$, implying that 
\[
d_G(v_7) \le \lvert V(G_7) \rvert -1 \le \frac{n-6}{2} -1 =\frac{n-8}{2}.
\]
Therefore 
\[
\sigma_4(G) \le 2+2+2+(\frac{n-8}{2}) = \frac{n+4}{2} < \frac{n+5}{2}.
\]
This is a contradiction (see graph $H_4$ in Example 4).

$\bullet$ For $k\ge7$, we have $\sigma_4(G) \ge \frac{4n-4k-7}{k+2}$ and therefore the total number of vertices in graph $G$ is 
\begin{align*}
n &= \sum_{j=1}^{k+2} \lvert V(G_j)\rvert = \sum_{j=1}^{m} \lvert V(G_j)\rvert + \sum_{j = m+1}^{k+2} \lvert V(G_j)\rvert \ge m + (k+2-m)(\sigma_4(G) - k) \\ 
&> 4+(k-4)(\sigma_4(G)-k) \ge 4+(k-4)\left(\frac{4n-4k-7}{k+2}-k\right).
\end{align*}
For $k=7, k=8, k=9$, we obtain 
\[
n \le 86, n\le \frac{218}{3}, n\le 74,
\]
respectively, a contradiction. For $k \ge 10$, we have
\[
n \le \frac{1}{3}k^2+\frac{8}{3}k+9+\frac{42}{k-6} < k^2+7k+18+\frac{33}{k-2}.
\]
This is also a contradiction.

\textbf{Case 5.1.2. }$\alpha(N) =2$. 

It can be seen that when $\Delta(N)=2$ and $\alpha(N) = 2$, we have $N \cong P_4$ or $N \cong 2P_2$. 

$\bullet$ For $k=3$, graph $G$ has $5$ components, of which $4$ belong to $S$. To form an independent set consisting of 4 vertices, we must choose 2 non-adjacent vertices in $S$ and 2 non-adjacent vertices both belonging to the remaining component outside $S$. Thus we have the estimation in this case $\sigma_4 (G) \le 1+2+(n-4-1-1)+(n-4-1)=2n-8<2n-7$, leading to a contradiction (see graph $H_1$ in Example 1).

$\bullet$ For $k\ge 4$, by the similar arguments as in Subcase 4.2, we have the estimation
\[
2n = 2\sum_{j=1}^{k+2} \lvert V(G_j)\rvert \ge 8 + (k-2)(\sigma_4(G) - k +1).
\]
For $k=4,k=5,k=7,k=8,k=9$, we obtain
\[
2n+2\le 2n, n \le 7, n\le \frac{373}{2}, n \le \frac{287}{2}, n \le \frac{829}{6},
\]
respectively, a contradiction. 

\noindent For $k=6$, the above inequality does not lead to a contradiction. We consider this case separately as follows. When $k=6$, we have $\lvert B \rvert=7$ and $G - B=\{G_i\}$, $i\in[8]$. From each component, we can choose a vertex $v_i \in V(G_i)$ for all $i \in [8]$ such that $v_i, i \in [4]$ is adjacent to at least one cut-edge and $N(v_i) \subseteq V(G_i)$ for all $i\in [8]\setminus [4]$. Without loss of generality, suppose that there exists a path in $G$ passing through 8 vertices $v_5, v_1, v_2,v_6, v_3, v_4, v_7,$ and $v_8$, in this order. Now obviously $\{v_1, v_3, v_5, v_6\}$ and $\{v_2, v_4, v_7, v_8\}$ are independent sets, so 
\begin{equation*}
\begin{cases}
    d_G(v_1) + d_G(v_3) + d_G(v_5) + d_G(v_6) \ge \sigma_4(G) = \frac{n+5}{2}, \\ 
    d_G(v_2) + d_G(v_4) + d_G(v_7) + d_G(v_8) \ge \sigma_4(G) = \frac{n+5}{2}.
\end{cases}
\end{equation*}
This implies $\displaystyle \sum_{i \in [8]} d_G(v_i) \ge 2\sigma_4(G) = n+5$. Nevertheless 
\[
    \sum_{i=1}^{8} d_G(v_i) \le \sum_{i=1}^{8}(\lvert V(G_i) \rvert-1) + (\lvert B \rvert+1) = (n-8)+8=n,
\]
a contradiction.

\noindent For $k \ge 10$, we have 
\[
n \le \frac{1}{2}k^2+\frac{9}{2}k+\frac{41}{2}+\frac{110}{k-6} < k^2+7k+18+\frac{33}{k-2}.
\]
This is also a contradiction.

\noindent\textbf{Case 5.2. } $\Delta(N) = 3$. 

For $\Delta(N) = 3$, the contracted graph $N$ is a tree with at least one vertex of degree 3. If $m \ge 7$, then there always exist four non-adjacent vertices in $S$ to form an independent set. This contradicts Claim 1. Therefore, we only need to consider the cases where $m \in \{4,5,6\}$. Moreover, since $N$ has at least 3 leaves, let $u_1, u_2, u_3$ be three vertices belonging to distinct components of $S$ corresponding to 3 leaves of $N$. Then, we obtain $\displaystyle \sum_{i \in [3]} d_G(u_i) \le \lvert B \rvert = k+1$.

There exists a vertex $v_j \in V(G_j)$ with $j \in [k+2]\setminus [m]$ such that $\{u_1,u_2,u_3,v_j\}$ forms an independent set. Then, $\displaystyle \sum_{i \in [3]} d_G(u_i) + d_G(v_j) \ge \sigma_4(G)$. This implies \[ 
    d_G(v_j) \ge \sigma_4(G) - \sum_{i \in [3]} d_G(u_i) \ge\sigma_4(G) - (k+1).
    \] 
Thus $\lvert V(G_j) \rvert \ge d_G(v_j)+1 \ge \sigma_4(G) - (k+1)+1 = \sigma_4(G) - k$ for all $j \in [k+2]\setminus [m]$.

 \textbf{Case 5.2.1. }$k=3$. For $m=4$, there is one component, say $G_5$, outside $S$ with $\lvert V(G_5) \rvert \ge \sigma_4(G) - k \ge (2n-7)-3 = 2n-10$, since $\sigma_4(G) \ge 2n-7$. The total number of vertices in $G$ can be expressed as
    \[
    n =\sum_{i=1}^{5} \lvert V(G_i)\rvert = \sum_{i=1}^{4} \lvert V(G_i)\rvert + \lvert V(G_5)\rvert \ge 4.1+(2n-10)=2n-6.
    \]
This implies $n\le 6$, contradicting $n \ge 8$
.

    \textbf{Case 5.2.2. }$k=4$. 

    For $m=4$, we have $n = \displaystyle \sum_{i \in [6]} \lvert V(G_i)\rvert \ge 4.1+2(n-4) = 2n-4 >n$, a contradiction.

    For $m=5$, we have $n = \displaystyle\sum_{i \in [6]} \lvert V(G_i)\rvert \ge 5.1+(n-4) = n+1$, a contradiction.

     \textbf{Case 5.2.3. }$k=5$. 

    For $m=4$, we have $n =\displaystyle \sum_{i \in [7]} \lvert V(G_i)\rvert \ge 4.1+3(n-5) = 3n-11>n$, a contradiction.

    For $m=5$, we have $n = \displaystyle\sum_{i \in [7]} \lvert V(G_i)\rvert \ge 5.1+2(n-5) = 2n-5>n$, a contradiction.

    For $m=6$, choose three leaves of $N$, then $\sigma_4(G) \le 1+2+2+(n-6-1)=n-2 <n-1$, a contradiction.

    \textbf{Case 5.2.4. }$k=6$.

    For $m=4$, we have $n = \displaystyle\sum_{i \in [8]} \lvert V(G_i)\rvert \ge 4.1+4\left(\frac{n+5}{2}-6\right) = 2n-10>n$, a contradiction.

    For $m=5$, we have $n = \displaystyle\sum_{i \in [8]} \lvert V(G_i)\rvert \ge 5.1+3\left(\frac{n+5}{2}-6\right) = \frac{3n-11}{2}>n$, a contradiction.

    For $m=6$, we have $\lvert V(G_7) \rvert \ge \sigma_4(G) - k \ge \frac{n+5}{2} - 6 = \frac{n-7}{2}$. Similarly, we also have $\lvert V(G_8) \rvert \ge \frac{n-7}{2}$. Therefore, $\displaystyle\sum_{i \in [6]} \lvert V(G_i) \rvert \le n - \frac{2(n-7)}{2} = 7$. Since $\Delta(N)=3$ and there does not exist four independent vertices in $N$, we deduce $\lvert V(G_i) \rvert=1$ for all $i \in [6]$. Let $V(G_i) = \{u_i\}$ for all $ i \in [6]$. Then $\lvert V(G_7) \rvert+\lvert V(G_8) \rvert=n-6$. Without loss of generality, assume $\lvert V(G_7) \rvert \le \frac{n-6}{2}$, implying 
    \[
    d_G(v_7) \le \lvert V(G_7) \rvert-1 = \frac{n-6}{2}-1=\frac{n-8}{2}.
    \]
    Since $\sigma_4(G) \ge \frac{n+5}{2}$, assume $u_1, u_2, u_3$ are three non-adjacent vertices in $S$, then 
    \[
    d_G(u_1) + d_G(u_2) + d_G(u_3) \ge \sigma_4(G) - d_G(v_7) \ge \frac{n+5}{2} - \frac{n-8}{2} =\frac{13}{2}.
    \]
    This is a contradiction because $\displaystyle\sum_{i \in [6]} d_G(u_i) \le 1+3+2+1+2+2=11 < 2\cdot\frac{13}{2}.$

    \textbf{Case 5.2.5. }$k \ge 7$. The total number of vertices in graph $G$ is
\begin{align*}
n &= \sum_{j=1}^{k+2} \lvert V(G_j)\rvert = \sum_{j=1}^{m} \lvert V(G_j)\rvert + \sum_{j=m+1}^{k+2} \lvert V(G_j)\rvert \ge m + (k+2-m)(\sigma_4(G)-k) \\ &> 4+(k-4)(\sigma_4(G)-k) \ge 4+(k-4)\left(\frac{4n-4k-7}{k+2}-k \right).
\end{align*}
For $k=7, k=8, k=9$, we obtain
\[
n \le 86, n\le \frac{218}{3}, n\le 74,
\]
a contradiction. For $k \ge 10$, we have 
\[
n\le \frac{1}{3}k^2+\frac{8}{3}k+9+\frac{42}{k-6} < k^2+7k+18+\frac{33}{k-2}.
\]
This is a contradiction.

    Five cases above all lead to contradictions. Thus the assumption is false, therefore $G$ has at most $k$ cut edges. This completes the proof of Theorem \ref{thm0.1}.
\end{proof}

We end this section by providing some examples to demonstrate that the conditions given in Theorem \ref{thm0.1} are best possible.
\begin{example}
	\upshape
	For $k=3$, let $H_1$ be a graph constructed by identifying an end vertex $u$ of a path $P_5$ with a vertex of a complete graph $K_{n-4}$, and deleting an edge $uv$ where $v \in V(K_{n-4}) \setminus \{u\}$ (see Figure 1). Then, $H_1$ is a graph of order $n$ and $\sigma_4(H_1)=1+2+(n-4-1-1)+(n-4-1)=2n-8<2n-7$. However, $H_1$ has $4$ cut-edges.
	\begin{center}
		\begin{tikzpicture}[
			scale=1.2,
			vertex/.style={circle, draw, fill=black, inner sep=1pt, minimum size=4pt},
			clique/.style={ellipse, draw, thick, minimum width=3cm, minimum height=2cm},
			edge/.style={thick},
			missing/.style={thick, dashed}
			]
			\node[clique] (K) at (4,0) {};
			\node at (4, 0.4) {$K_{n-4}-uv$};
			\node[vertex, label={[xshift=-8pt]above:$u$}] (u) at (K.west) {};
			\node[vertex, label=right:$v$] (v) at (4.5, 0) {};
			\node[vertex] (x4) at (1.2, 0) {};
			\node[vertex] (x3) at (-0.1, 0) {};
			\node[vertex] (x2) at (-1.4, 0) {};
			\node[vertex] (x1) at (-2.7, 0) {};
			\draw[edge] (x1) -- (x2);
			\draw[edge] (x2) -- (x3);
			\draw[edge] (x3) -- (x4);
			\draw[edge] (x4) -- (u);
			\draw[missing] (u) -- (v);
			\node at ($(u)!0.5!(v)$) {\small $\times$};
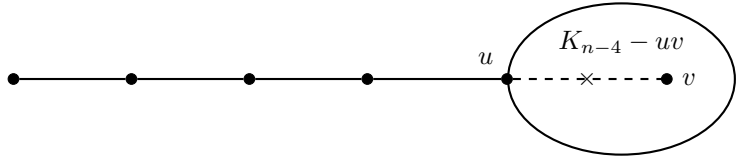
			\node[below=1.5cm] at (1,0) {\textbf{Figure 1.} Graph $H_1$ for $k=3$.};
		\end{tikzpicture}
	\end{center}
\end{example}

\begin{example}
	\upshape
	For $k=4$, let $H_2$ be a graph constructed by identifying an end vertex of a path $P_6$ with a vertex of a complete graph $K_{n-5}$ (see Figure 2). Then, $H_2$ is a graph of order $n$ and $\sigma_4(G_2)=1+2+2+(n-5-1)=n-1 < n$. However, $H_2$ has $5$ cut-edges.
	\begin{center}
		\begin{tikzpicture}[
			scale=1.2,
			vertex/.style={circle, draw, fill=black, inner sep=1pt, minimum size=3pt},
			clique/.style={ellipse, draw, thick, minimum width=2.5cm, minimum height=1.25cm},
			edge/.style={thick}
			]
			\node[clique] (K) at (7,0) {};
			\node at (7,0) {\large $K_{n-5}$};
			\node[vertex] (v5) at (K.west) {};
			\node[vertex] (v4) at ($(v5)+(-1.5,0)$) {};
			\node[vertex] (v3) at ($(v5)+(-3.0,0)$) {};
			\node[vertex] (v2) at ($(v5)+(-4.5,0)$) {};
			\node[vertex] (v1) at ($(v5)+(-6.0,0)$) {};
			\node[vertex] (v0) at ($(v5)+(-7.5,0)$) {};
			\draw[edge] (v0) -- (v1) node[midway, below] {};
			\draw[edge] (v1) -- (v2) node[midway, below] {};
			\draw[edge] (v2) -- (v3) node[midway, below] {};
			\draw[edge] (v3) -- (v4) node[midway, below] {};
			\draw[edge] (v4) -- (v5) node[midway, below] {};
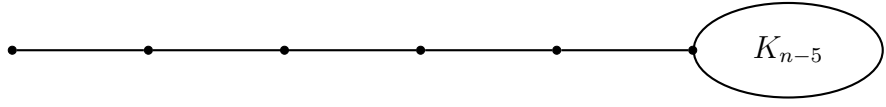
			\node[below=1cm] at (3, 0) {\textbf{Figure 2.} Graph $H_2$ for $k=4$.};
		\end{tikzpicture}
	\end{center}
\end{example}

\begin{example}
	\upshape
	For $k=5$, let $H_3$ be a graph constructed by identifying an end vertex of a path $P_7$ with a vertex of a complete graph $K_{n-6}$ (see Figure 3). Then, $G_3$ is a graph of order $n$ and $\sigma_4(H_3)=1+2+2+(n-6-1)=n-2 < n-1$. However, $H_3$ has $6$ cut-edges.
	\begin{center}
		\begin{tikzpicture}[
			scale=1.1,
			vertex/.style={circle, draw, fill=black, inner sep=1pt, minimum size=3pt},
			clique/.style={ellipse, draw, thick, minimum width=2.5cm, minimum height=1.25cm},
			edge/.style={thick}
			]
			\node[clique] (K) at (9,0) {};
			\node at (9,0) {\large $K_{n-6}$};
			\node[vertex] (v6) at (K.west) {};
			\node[vertex] (v5) at ($(v6)+(-1.4,0)$) {};
			\node[vertex] (v4) at ($(v5)+(-1.4,0)$) {};
			\node[vertex] (v3) at ($(v4)+(-1.4,0)$) {};
			\node[vertex] (v2) at ($(v3)+(-1.4,0)$) {};
			\node[vertex] (v1) at ($(v2)+(-1.4,0)$) {};
			\node[vertex] (v0) at ($(v1)+(-1.4,0)$) {};
			\draw[edge] (v0) -- (v1);
			\draw[edge] (v1) -- (v2);
			\draw[edge] (v2) -- (v3);
			\draw[edge] (v3) -- (v4);
			\draw[edge] (v4) -- (v5);
			\draw[edge] (v5) -- (v6); 
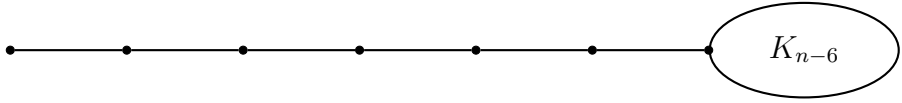
			\node[below=1cm] at (4.5, 0) {\textbf{Figure 3.} Graph $H_3$ for $k=5$.};      
		\end{tikzpicture}
	\end{center}
\end{example}

\begin{example}
	\upshape
	For $k=6$, let $H_4$ be a graph constructed by identifying each end vertex of a path $P_8$ with a complete graph $K_{\frac{n-6}{2}}$ (see Figure 4). Then, $H_4$ is a graph of order $n$ and $\sigma_4(H_4)=2+2+2+(\frac{n-6}{2}-1)=\frac{n+4}{2} < \frac{n+5}{2}$. However, $H_4$ has $7$ cut edges.
	\begin{center}
		\begin{tikzpicture}[
			scale=1.0,
			vertex/.style={circle, draw, fill=black, inner sep=1pt, minimum size=3pt},
			clique/.style={ellipse, draw, thick, minimum width=2.5cm, minimum height=1.25cm},
			edge/.style={thick}
			]
			\node[clique] (L) at (0,0) {};
			\node at (0,0) {\large $K_{\frac{n-6}{2}}$};
			\node[clique] (R) at (10.5,0) {}; 
			\node at (10.5,0) {\large $K_{\frac{n-6}{2}}$};
			\node[vertex] (v1) at (L.east) {};
			\node[vertex] (v8) at (R.west) {};
			\foreach \i in {2,3,4,5,6,7} {
				\path (v1) -- (v8) node[pos={(\i-1)/7}, vertex] (v\i) {};
			}
			\draw[edge] (v1) -- (v2);
			\draw[edge] (v2) -- (v3);
			\draw[edge] (v3) -- (v4);
			\draw[edge] (v4) -- (v5);
			\draw[edge] (v5) -- (v6);
			\draw[edge] (v6) -- (v7);
			\draw[edge] (v7) -- (v8);
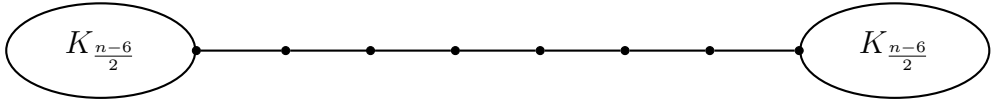
			\node[below=1cm] at (5, 0) {\textbf{Figure 4.} Graph $H_4$ for $k=6$.};    
		\end{tikzpicture}
	\end{center}
\end{example}

\begin{example}
	\upshape
	For $k \ge 7$. Given any two integers $k\ge7, m\ge k+1,$ let $D_{i}$ be $k+2$ copies of the complete graph $K_{m}$. We construct a connected graph $H_5$ by adding edges connecting each vertex $v\in V(D_{1})$ to exactly one vertex $u_{i}\in V(D_{i})$ where $i\in[k+2]\setminus\{1\}$ (see Figure 5). It is easy to see that the order of $H_5$ is $n=(k+2)m$ and $\sigma_{4}(H_5)=4(m-1)=\frac{4n-4k-8}{k+2} < \frac{4n-4k-7}{k+2}$. However, there are $k+1$ cut edges in $H_5$.
\end{example}
\begin{center}
	\begin{tikzpicture}[
		scale=0.9,
		vertex/.style={circle, draw, fill=black, inner sep=1pt, minimum size=3pt},
		clique/.style={circle, draw, thick, minimum size=2cm},
		satellite/.style={circle, draw, thick, minimum size=1.3cm},
		edge/.style={thick},
		cutedge/.style={thick} 
		]
		\node[clique] (D1) at (0,0) {};
		\node at (0,0) {$D_1 \cong K_m$};
		\begin{scope}[shift={(0, 4)}]
			\node[satellite] (D2) at (0,0) {};
			\node at (0, 0) {$D_2$};
			\node[vertex] (u2) at (D2.south) {};
		\end{scope}
		\begin{scope}[shift={(160:4)}]
			\node[satellite] (D3) at (0,0) {};
			\node at (0, 0) {$D_3$};
			\node[vertex] (u3) at (D3.east) {}; 
		\end{scope}
		\begin{scope}[shift={(20:4)}]
			\node[satellite] (Dk) at (0,0) {};
			\node at (0, 0) {$D_{k+2}$};
			\node[vertex] (uk) at (Dk.west) {}; 
		\end{scope}
		\node[vertex] (v2) at (D1.north) {};
		\draw[cutedge] (v2) -- (u2) node[midway, right, black] {};
		\node[vertex] (v3) at (D1.160) {};
		\draw[cutedge] (v3) -- (u3);
		\node[vertex] (vk) at (D1.20) {};
		\draw[cutedge] (vk) -- (uk);
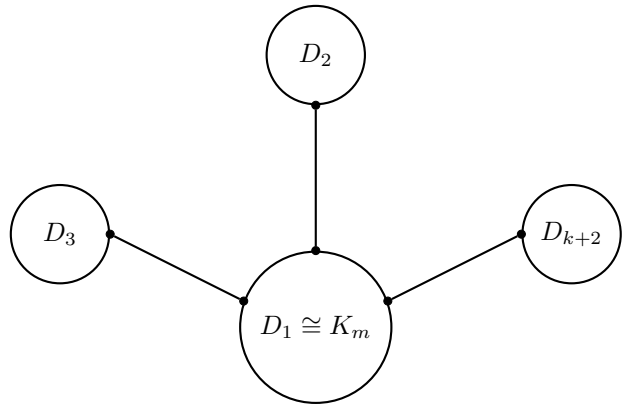
		\node[below=1.5cm] at (0, 0) {\textbf{Figure 5.} Graph $H_5$ for $k\ge 7$.};    
	\end{tikzpicture}
\end{center}

\section{$cfc(G)$ and the degree sum}
To begin with, we present several results concerning the conflict-free connection number that will be used to establish our main findings. Czap et al. \cite{Czap} determined the conflict-free connection number for $2$-connected graphs.
\begin{lemma}[Czap et al. \cite{Czap}]
\label{lem_cfc(G)=2-2-connected}
If $G$ is a non-complete, $2$-connected graph, then $cfc(G)=2$.
\end{lemma}

The authors in \cite{Chang2019} generalized the result of Lemma \ref{lem_cfc(G)=2-2-connected} to $2$-edge-connected graphs, as stated in the following theorem. 

\begin{corollary}[Chang et al. \cite{Chang2019}]
\label{cor_cfc(G)=2-2-edge-connected}
Let $G$ be a non-complete, $2$-edge-connected graph, then $cfc(G) = 2$.
\end{corollary}

Next, we present several results on the conflict-free connection number under various conditions involving cut-edges. Recall that $C(G)$ denotes the subgraph of $G$ induced by its cut-edges and $\vert C(G)\vert$ denotes the number of edges in $C(G)$. Suppose $C(G)$ is a linear forest consisting of $p$ ($p \geq 0$) components $Q_1,Q_2,\ldots, Q_p$ with $n_i=\vert V(Q_i) \vert$ and $2\leq n_1\leq n_2\leq \ldots\leq n_p$. 

\begin{theorem}[Chang et al. \cite{Chang2018}]
\label{cfc=2-chang2018}
If $G$ is a connected, non-complete graph with $C(G)$ being a linear forest with $2=n_1=n_2=\ldots =n_{p-1}\leq n_p\leq 4$ or $C(G)$ being edgeless, then $cfc(G)=2$.
\end{theorem}
 
Without requiring that $C(G)$ is a linear forest, then an upper bound for the conflict-free connection number of an arbitrary connected graph has been shown by Ji et al. \cite{Ji-preprint}.
\begin{theorem}[Ji et al. \cite{Ji-preprint}]
\label{upper_bound_cfc(G)_C(G)}
If $G$ is a connected, non-complete graph with $C(G)$, then $cfc(G)\leq \max\lbrace 2, \vert C(G)\vert\rbrace$.
\end{theorem}

Moreover, let  $\lbrace C_1, C_2,\ldots, C_m\rbrace$ be components of $C(G)$, where $m\geq 1.$ Set $h(G)=\max\{cfc(C_i) \mid i\in [m]\}.$ In \cite{Czap}, the authors proved that 
\begin{theorem}[Czap et al. \cite{Czap}]
	\label{cfc=cut-edges}
	If $G$ is a connected graph with cut-edges, then $h(G)\leq cfc(G)\leq h(G)+1$.
\end{theorem}

We are now able to prove our results. Recall Theorem \ref{main0} here.
\begin{theorem-non}{Theorem \ref{main0}} 
Let $G$ be a connected graph of order $n$ and let $k$ be an integer. If
\[
\sigma_{4}(G) \ge
\begin{cases} 
	2n-7 &\text{for } k=3 \text{ and } n\ge 8 \\ 
	n &\text{for } k =4 \text{ and } n\ge9\\
	n-1 &\text{for } k=5 \text{ and } n \ge 11 \\
	\displaystyle\frac{n+5}{2} &\text{for } k=6 \text{ and } n \ge 17 \\
	\displaystyle\frac{4n-4k-7}{k+2} & \text{for } k \ge 7 \text{ and }\begin{cases} 
		n \ge 286 &\text{for } k=7 \\
		n \ge 207 &\text{for } k=8 \\
		n \ge 190 &\text{for } k=9 \\
		\displaystyle n > k^2 + 7k + 18 + \frac{33}{k-2}  &\text{for } k\ge 10, \\
	\end{cases}
\end{cases}
\]
then $cfc(G)\leq k.$
\end{theorem-non}
\begin{proof}
	\upshape
	Since the hypotheses are satisfied, $G$ has at most $k$ cut-edges by Theorem \ref{thm0.1}. Thanks to Theorem \ref{upper_bound_cfc(G)_C(G)} we conclude that $cfc(G)\leq k.$
	\end{proof}

\noindent{\bf Sharpness.}\, We only consider sharpness for $k \ge 7.$ Given any two integers $k\ge7$ and $m\ge k+1,$ let $M_{i}$ be $k+2$ copies of the complete graph $K_{m}$ and fix $v\in V(M_1).$ We construct a connected graph $M$ by adding edges connecting the vertex $v$ to exactly one vertex $u_{i}\in V(M_{i})$ where $i\in[k+2]\setminus\{1\}$ (see Figure 6). It is easy to see that the order of $M$ is $n=(k+2)m$ and $\sigma_{4}(M)=4(m-1)=\frac{4n-4k-8}{k+2} < \frac{4n-4k-7}{k+2}$. However, we obtain $cfc(G)\geq k+1$. Therefore, the results of Theorem \ref{main0} are optimal for every $k\geq 7.$

\begin{center}
	\begin{tikzpicture}[
		scale=0.9,
		vertex/.style={circle, draw, fill=black, inner sep=1pt, minimum size=3pt},
		clique/.style={circle, draw, thick, minimum size=2cm},
		satellite/.style={circle, draw, thick, minimum size=1.3cm},
		edge/.style={thick},
		cutedge/.style={thick} 
		]
		\node[clique] (M1) at (0,0) {};
		\node at (0,0) {$M_1 \cong K_m$};
		\begin{scope}[shift={(0, 4)}]
			\node[satellite] (M2) at (0,0) {};
			\node at (0, 0) {$M_2$};
			\node[vertex] (u2) at (M2.south) {};
		\end{scope}
		\begin{scope}[shift={(160:4)}]
			\node[satellite] (M3) at (0,0) {};
			\node at (0, 0) {$M_3$};
			\node[vertex] (u3) at (M3.east) {}; 
		\end{scope}
		\begin{scope}[shift={(20:4)}]
			\node[satellite] (Mk) at (0,0) {};
			\node at (0, 0) {$M_{k+2}$};
			\node[vertex] (uk) at (Mk.west) {}; 
		\end{scope}
		\node[vertex] (v2) at (M1.north) {};
		\draw[cutedge] (v2) -- (u2) node[midway, right, black] {};
		\node[vertex] (v3) at (M1.north) {};
		\draw[cutedge] (v3) -- (u3);
		\node[vertex] (vk) at (M1.north) {};
		\draw[cutedge] (vk) -- (uk);
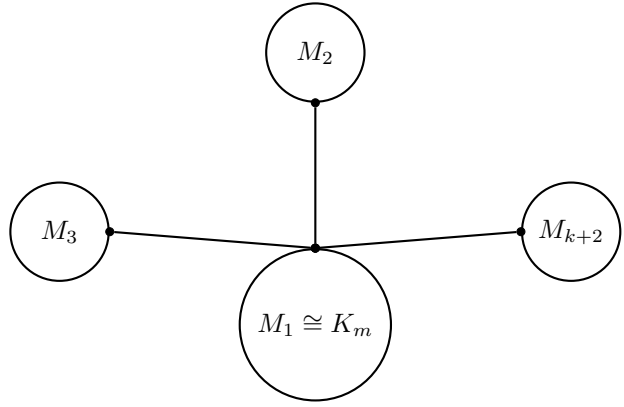
		\node[below=1.5cm] at (0, 0) {\textbf{Figure 6.} Graph $M$ for $k\ge 7$.};    
	\end{tikzpicture}
\label{F5}
\end{center}

Recall Theorem \ref{main1} here.
\begin{theorem-non}{Theorem \ref{main1}}\label{cfc(G)=2-sigma_4(G)}
Let $n\geq 8$ be an integer and $G$ be a connected, non-complete graph of order $n$. If $\sigma_4(G)\geq 2n-7$, then one of the following holds:
\begin{enumerate}
	\item $C(G)\cong K_{1,3}$ implying $cfc(G)=3$.
	\item $C(G)$ is a linear forest implying $cfc(G)=2$.
\end{enumerate}
\end{theorem-non}
\begin{proof}
Since $\sigma_4(G)\geq 2n-7$, by Theorem \ref{thm0.1}, $G$ has at most three cut-edges. 

For 1). By Theorem \ref{upper_bound_cfc(G)_C(G)}, $cfc(G)\leq 3$. Since $C(G)\cong K_{1,3}$, it can be readily seen that $cfc(G)\geq 3$. Hence, $cfc(G)=3$.

For 2). Since $G$ is a connected, non-complete graph and $C(G)$ is a linear forest with at most three cut-edges, applying Theorem \ref{cfc=2-chang2018}, $cfc(G)=2$. 

Our proof is obtained. 	
\end{proof}
\noindent\textbf{Sharpness.} Graph $H_1$ in Example 1 is a connected, non-complete graph with $\sigma_4(H_1)=2n-8$ and $C(H_1)$ is a linear forest. Clearly, by Theorem \ref{cfc=2-chang2018}, $cfc(H_1)\geq 3$. 

Next, recall Theorem \ref{main2} here. 
\begin{theorem-non}{Theorem \ref{main2}}\label{cfc=2-sigma_4(G)-min}
Let $n\geq9$ be an integer and $G$ be a connected, non-complete graph of order $n$ and $\delta(G)\geq2$. If $\sigma_4(G)\geq n$, then $cfc(G)=2$.
\end{theorem-non}
\begin{proof}
    By Theorem \ref{thm0.1}, $G$ has at most $4$ cut-edges. Let $G_1,\ldots,G_k$ be $k$ components in $G - C(G)$, where $k\leq5$. Let $B=\lbrace G_i\in G - C(G)\text{ where there exists } v_i\in V(G_i) \text{ such that } N(v_i)\subseteq V(G_i)\rbrace$. Since $\delta(G)\geq 2$, it follows that $\vert B\vert \geq1$. 

    Clearly, $\vert B\vert\leq 3$. Otherwise, let $v_j\in V(G_j)$ such that $N(v_j)\subseteq V(G_j)$, where $j\in[4]$. Now, 
    $$n\geq \sum_{j=1}^4\vert V(G_j)\vert\geq\sum_{j=1}^4 (d_G(v_j)+1)\geq\sigma_4(G)+4\geq n+4,$$
    a contradiction. Hence, $1\leq \vert B\vert\leq3$. Moreover, $C(G)\notin \lbrace K_{1,4}, K_{1,3}\cup P_2\rbrace$.

    Now, suppose that $cfc(G)\geq3$. By Lemma \ref{lem_cfc(G)=2-2-connected}, Corollary \ref{cor_cfc(G)=2-2-edge-connected} and Theorem \ref{cfc=2-chang2018}, $C(G)\in\lbrace{2P_3}, K_{1,3}\rbrace$. Since $\delta(G)\geq 2$, $\vert B\vert =3$. Let $B=\lbrace G_1, G_2, G_3\rbrace$ and $v_j\in V(G_j)$  such that $N(v_j)\subseteq V(G_j)$, where $j\in[3]$. We consider two following cases. 
    
    \noindent\textbf{Case 1.} $C(G)\cong 2P_3$. Let $u_1u_2u_3$ and $u_4u_5u_6$ are two paths $P_2$ in $C(G)$, where $u_1\in V(G_1)$, $u_3,u_4\in V(G_2)$ and $u_6\in V(G_3)$. It can be readily seen that $\lbrace v_1,v_2,v_3,u_2\rbrace$ is an independent set in $G$. Hence, 
    $$
    n=\sum_{j=1}^3\vert V(G_j)\vert + 2\geq\sum_{j=1}^3(d_G(v_j)+1)+d_G(u_2)\geq \sigma_4(G)+3\geq n+3,
    $$
    a contradiction. 

   \noindent\textbf{Case 2.} $C(G)\cong K_{1,3}$. Let $K_{1,3}=uu_1u_2u_3$, where $u$ is the center vertex of $K_{1,3}$ and $v_i\in V(G_i)$ for each $i \in [3]$. Hence, $\lbrace u,v_1,v_2,v_3\rbrace$ is an independent set in $G$. Now,
   $$
     n=\sum_{j=1}^3\vert V(G_j)\vert + 1\geq\sum_{j=1}^3(d_G(v_j)+1)+d_G(u)-2\geq \sigma_4(G)+1\geq n+1,
   $$
   a contradiction. 

   Therefore, $cfc(G)=2$. We finish our proof. 
\end{proof}

Next, recall Theorem \ref{main3} here.
\begin{theorem-non}{Theorem \ref{main3}}\label{cfc(G)=2-sigma_4(G)-natural-ext}
Let $n\geq 21$ be an integer and $G$ be a connected, non-complete graph of order $n$ and $\delta(G)\geq 3$. If $C(G)$ is a linear forest and $\sigma_4(G)\geq\frac{4n-19}{5}$, then $cfc(G)=2$.
\end{theorem-non}

\begin{proof} 
For $n\geq21$ we obtain $\frac{4n-19}{5}\geq\frac{n+5}{2}$. Hence $\sigma_4(G)\geq\frac{n+5}{2}$. By Theorem \ref{thm0.1}, $G$ has at most $6$ cut-edges. 

Suppose that $4\leq\vert C(G)\vert\leq6$, then $G - C(G)$ has exact $k$ components, say $G_1,\ldots,G_k$, where $5\leq k\leq 7$. Let $n_j$ be the order of $G_j$, where $j\in[k]$. Since $\delta(G)\geq 3$ and $C(G)$ is a linear forest, for every index $j\in[5]$, there always exits a vertex $v_j\in V(G_j)$ such that $N(v_j)\subseteq V(G_j)$. Following Case 1 in the proof of Theorem \ref{thm0.1}, we get
$$
4n=4\sum_{j=1}^{k}\vert V(G_j)\vert \geq 4\sum_{j=1}^{k} (d_G(v_j)+1)\geq k\sigma_4(G)+4k\geq k\left(\frac{4n-19}{5}\right)+4k. 
$$
Note that $5\leq k\leq7$. After some simple manipulations, we get a contradiction. Now, $G$ has at most $3$ cut-edges. Moreover, $C(G)$ is linear forest. By Theorem \ref{cfc=2-chang2018}, $cfc(G)=2$.

We obtain our proof.
\end{proof}
\noindent\textbf{Sharpness.} Let $n = 5p$, for $p \ge 4$. Consider a graph $G$ consisting of five disjoint blocks $G_1, G_2, G_3, G_4, G_5$, where each block isomorphics to a complete graph $K_p$. Choose the vertices $v_1 \in V(G_1)$, $v_2 \in V(G_2)$, $v_4 \in V(G_4)$, $v_5 \in V(G_5)$, and choose two vertices $u_3, v_3 \in V(G_3)$. The graph $G$ is connected by adding four following cut-edges $v_1v_2$, $v_2u_3$, $v_3v_4$, and $v_4v_5$. Note that $u_3$ is not necessary different from $v_3$. In fact, if $u_3$ and $v_3$ coincide, then $C(G) \cong P_5$. Otherwise, $C(G) \cong 2P_3$. We obtain $\sigma_4(G) = 4(p-1) = 4(\frac{n}{5}-1) = \frac{4n-20}{5}$. However, it can be witnessed that $cfc(G) \ge 3$. Therefore, the condition $\sigma_4(G)\geq\frac{4n-19}{5}$ in Theorem \ref{main3} is tight.
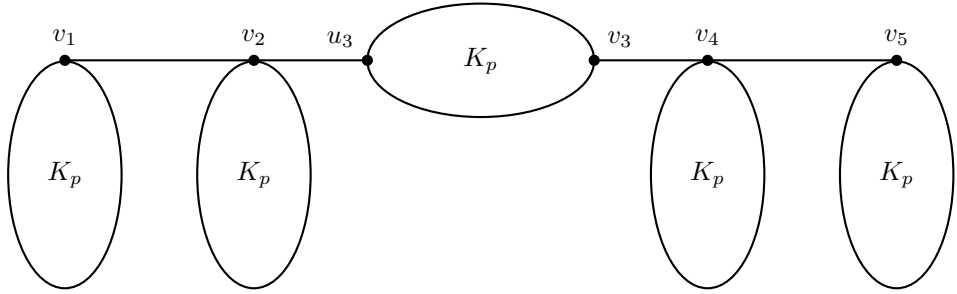
\begin{center}
    \begin{tikzpicture}[
    every node/.style={font=\normalsize},
    vclique/.style={ellipse, draw=black, thick, minimum width=1.5cm, minimum height=3cm, text=black},
    hclique/.style={ellipse, draw=black, thick, minimum width=3cm, minimum height=1.5cm, text=black},
    cutnode/.style={circle, draw=black, thick, fill=black, inner sep=1.2pt}
]
    \coordinate (V1) at (0, 0);
    \coordinate (V2) at (2.5, 0);
    \coordinate (V3) at (4, 0);
    \coordinate (V4) at (7, 0); 
    \coordinate (V5) at (8.5, 0);
    \coordinate (V6) at (11, 0);

    \draw[thick] (V1) -- (V2);
    \draw[thick] (V2) -- (V3);
    \draw[thick] (V4) -- (V5);
    \draw[thick] (V5) -- (V6);
    \node[vclique, anchor=north] (G1) at (V1) {$K_p$};
    \node[vclique, anchor=north] (G2) at (V2) {$K_p$};
    \node[hclique] (G3) at ($(V3)!0.5!(V4)$) {$K_p$};
    \node[vclique, anchor=north] (G4) at (V5) {$K_p$};
    \node[vclique, anchor=north] (G5) at (V6) {$K_p$};
    \node[cutnode, label=$v_1$] at (V1) {};
    \node[cutnode, label=$v_2$] at (V2) {};
    \node[cutnode, label=135:$u_3$] at (V3) {};
    \node[cutnode, label=45:$v_3$] at (V4) {};
    \node[cutnode, label=$v_4$] at (V5) {};
    \node[cutnode, label=$v_5$] at (V6) {};
\end{tikzpicture}
\end{center}
 \section{Discussion}
 Let $n, l$ and $k$ represent arbitrary positive integers. We use $g(k)$ to denote a function of $k$, and $f(n, k)$ to denote a function of both $n$ and $k$. In \cite{Doan2024}, the authors proposed the following problem.
 
 \noindent{\bf Problem 1.}
 	{\it
 	Let $k\geq k_0$ and $n\geq g(k)$ be two integers. If $G$ is a connected graph of order $n$ and $\sigma_l(G)\geq f(n,k)$, then $G$ has at most $k$ cut-edges.
 }

 As we mentioned earlier, we believe that our technique in this paper can resolve this one. Thus, we hope the subsequent problem can be addressed.

\noindent{\bf Problem 2.} {\it Let $G$ be a connected graph with order $n$ and let $k$ be an integer ($k\geq 5$). How can we find a function $f(n,k)$ such that if $\sigma_{l}(G)\geq f(n,k),$ then $cfc(G)\leq k$ for every  $n\geq g(k).$ 
}

 On the other hand, Ha et al. \cite{HHM} have recently studied the conflict-free connection number for planar graphs. Hence, we propose a question for the final discussion. 
 
\noindent{\bf Problem 3.} {\it Can we tell some similar results about the planar graphs?
}

\end{document}